\documentclass{amsart}
\usepackage{amssymb,latexsym}
\usepackage{amsmath}
\usepackage{graphicx}
\usepackage{amscd}
\usepackage{color}
\usepackage{enumerate}

\numberwithin{equation}{section}
\theoremstyle{plain}
 \newtheorem{theorem}{Theorem}[section]
 \newtheorem{lemma}[theorem]{Lemma}

 \newtheorem{corollary}[theorem]{Corollary}
\theoremstyle{definition}

 \newtheorem{example}[theorem]{Example}

\theoremstyle{remark}
 
\newcommand\atoms [1]{\textup{At}(#1)}
\newcommand\coatoms [1]{\textup{CA}(#1)}
\newcommand\datoms [1]{\textup{At}'(#1)}
\newcommand\dcoatoms [1]{\textup{CA}'(#1)}
\newcommand\ideal [1]  {\mathord{\downarrow}#1}
\newcommand\filter [1]  {\mathord{\uparrow}#1}
\newcommand \ourn  {\hat L_0}
\newcommand \aourn  {\hat {\alg L}_0}
\newcommand \oure  {\hat L}
\newcommand \aoure  {\hat {\alg L}}
\newcommand \adbgad {\alg G^{\sst{\dblesign}}}
\newcommand \minne [1] {#1^{\mathord{-}01}}
\newcommand \cordiakern {\kern 0.58cm}
\newcommand \bdiakern {\kern 0.1cm}

\newcommand \plusskip{\kern 5 pt}
\newcommand \minusskip{\kern -8 pt}

\newcommand \celfunct {{F_{\kern-1.5pt  \tsssty{com}}}}
\newcommand \ccelfunct [1] {{F_{\kern-1.5pt  \tsssty{com}}^{#1} }}
\newcommand \posetfunct {{F_{\kern-1.5pt \tsssty{pos}}}}

\newcommand \psnv [1] {{\overline #1}\kern 1pt'}

\newcommand \fofunct {{G_{\kern-1.5pt \tsssty{forg}}}}
\newcommand \compfunct {{G_{\kern-1.5pt \tsssty{prod}}}}
\newcommand \foifunct {{G_{\kern-1.5pt \tsssty{forg}}^{-1}}}

\newcommand \xceltransf [1] {{\boldsymbol\pi}_{\kern-0.5pt #1}^{\tsssty{com}}}

\newcommand \xpceltransf [2] {{{\boldsymbol\pi}^{\tsssty{com}}_{\kern-0.5pt #1}(#2)}}

\newcommand \trivcom {\vec v^{\kern 1pt\textup{triv}}}

\newcommand \tsssty[1]{{\scriptscriptstyle{\textup{#1}}}}

\newcommand \pairs [1] {{\textup{Pairs}^{\leq}(#1)}}
\newcommand \ntpairs [1] {{\textup{Pairs}^{<}(#1)}}
\newcommand \covpairs [1] {{\textup{Pairs}^{\prec}(#1)}}

\newcommand\dblesign{\textup{db}}

\renewcommand\rho{\varrho}

\newcommand \sst [1] {\scriptscriptstyle #1}
\newcommand \fprinc {\textup{Princ}}
\newcommand \princ[1] {\fprinc(#1)}
\newcommand \aut[1] {\textup{Aut}(#1)}

\newcommand \cg[2] {\textup{con}(#1,#2)}

\newcommand \cgi[3] {\textup{con}_{#1}(#2,#3)}
\newcommand\alg [1] {{\mathcal #1}}

\newcommand \blokk[2] {#1/#2}

\newcommand\iideal[2]{\mathord\downarrow_{\kern-2pt #1\kern 1pt} #2}
\newcommand\ifilter[2]{\mathord\uparrow_{\kern-2pt #1\kern 1pt} #2}
\newcommand \tuple [1] {\langle #1\rangle}
\newcommand \pair [2] {\tuple{#1,#2}}

\newcommand\red[1]{{\textcolor{red}{#1}}}
\newcommand \tbf [1] {\textbf{#1}} 
\newcommand \set[1] {\{#1\}}

\newcommand \nablaell [1] {\nabla_{\kern -2pt #1}}
\newcommand\czinit [1] {} 
\newcommand\init [1] {} 

\begin{document}
\title[Representing posets and groups by lattices]
{A simultaneous representation of a group and a bounded poset with lattice automorphisms and  principal congruences}

\author[G.\ Cz\'edli]{G\'abor Cz\'edli}
\email{czedli@math.u-szeged.hu}
\urladdr{http://www.math.u-szeged.hu/\textasciitilde{}czedli/}
\address{University of Szeged, Bolyai Institute, 
Szeged, Aradi v\'ertan\'uk tere 1, HUNGARY 6720}

\dedicatory{The changes made after the first version, \textup{ http://arxiv.org/pdf/1508.04302v1} , are  in \red{red}.}

\thanks{This research was supported by
NFSR of Hungary (OTKA), grant number K 115518}

\subjclass[2000] {06B10
\red{.\hfill August 21, 2015}}

\keywords{Principal lattice congruence,  lattice automorphism, poset, group,  simultaneous representation}

\begin{abstract} 
Given a poset $P$ with at least two elements and a group  $G$, there exists a selfdual lattice of length 16 such 
that the collection of its principal congruences is order isomorphic to $P$ while
its automorphism group to  $G$.
\end{abstract}

\maketitle
\section{Introduction}
\subsection{Aim} For a bounded lattice $L$, $\princ L=\tuple{\princ L;\subseteq}$ denotes the poset of principal congruences of $L$. It is a bounded ordered set. 
Conversely, \init{G.~}Gr\"atzer~\cite{ggprincl} proved that every bounded ordered set is isomorphic to $\princ L$ for an appropriate lattice $L$ of length 5; see \cite{czgprincc} for a generalization to the countably infinite case.
Let $\aut L=\tuple{\aut L;\circ}$ stand for the group of automorphisms of $L$. By \init{G.~}Birkhoff~\cite{birkhoff},  every group is isomorphic to $\aut L$ for an appropriate lattice $L$. Our goal is to prove the following theorem.

\begin{theorem}\label{thmmain} If $P$ is a bounded ordered set with at least two elements and  $G$ is  an arbitrary group, then there exists a selfdual lattice $L$ of length sixteen such that $\princ L$ and $\aut L$ are isomorphic to $P$ and $G$, respectively. 
\end{theorem}

\subsection{Sketch}\label{subsectkdldl}
For those familiar with \cite{czginjlatcat}, the next  paragraph and  Examples~\ref{exegy}--\ref{examplehrm}, see later, are sufficient to understand our construction and the idea of the proof. 

Each of \init{G.~}Gr\"atzer~\cite{ggprincl} and  \cite{gghomoprinc} and 
\czinit{G.~}Cz\'edli~\cite{czgprincc}, \cite{czgsingleinjectiveprinc}, and \cite{czginjlatcat}
associates a lattice $L$ with $P$ such that $\princ L\cong P$. In  these papers, we start with a set $\set{[a_p,b_p]: 0\neq p\in P}$ of ``key'' prime intervals, and add certain additional elements, which are organized into ``gadgets'', to obtain an appropriate $L$. Here, to get rid of
the automorphisms inherited from $P$, we replace
the key prime intervals with distinct simple bounded lattices that have no nontrivial automorphism. These lattices are constructed in Section~\ref{sectrigid}. Next, the result from \init{G.~}Sabidussi~\cite{sabidussi} allows us to represent $G$ as the automorphism group of a graph $\tuple{V;E}$.  For each $v\in V$, we add a prime interval $[a_v,b_v]$ to our lattice together with appropriate gadgets to force that these new prime intervals generate the largest congruence. (Later, to make these intervals recognizable, we enlarge them to disjoint copies of an approrpiate simple lattice.) Whenever $\pair u v\in E$, we add a gadget between $[a_u,b_u]$ and $[a_v,b_v]$. The  new gadgets encode the graph into the lattice without changing $\princ L$. These details are elaborated in Section~\ref{sectconstcompl}, where  both the quasi-coloring technique developed in 
\cite{czgrepres}--\cite{czginjlatcat} and the ideas of
 \cite{czgprincc}--\cite{czginjlatcat}  and \init{G.~}Gr\"atzer~\cite{ggprincl}
are intensively used; however, it suffices  if the reader only keeps  \cite{czginjlatcat} nearby.

\section{Graphs and rigid simple lattices}\label{sectrigid}
By a graph we mean a pair $\tuple{V;E}$ where $V$ is a nonempty set, the \emph{vertex set}, and $E$ is a subset of the set of two-element subsets of $E$, the \emph{edge set} of the graph.
The following statement is due to \init{G.~}Sabidussi~\cite{sabidussi}; see also \init{R.~}Frucht~\cite{fruchtgnul} and \cite{fruchtgegy} for the finite case.

\begin{lemma}[\cite{sabidussi}]\label{lemmasabidussi}
 For every group $G$, there exists a graph $\tuple{V;E}$ such that $G$ is isomorphic to $\aut{\tuple{V;E}}$.
\end{lemma}

Next, we  borrow some concepts from \cite{czginjlatcat}. A \emph{quasiorder} is a  reflexive transitive relation. 
For a lattice or ordered set  $L=\tuple{L;\leq}$ and $x,y\in L$, $\pair x y$ is called an \emph{ordered pair} of $L$ if $x\leq y$. If $x=y$, then  $\pair x y$  is a \emph{trivial ordered pair}. The set of ordered pairs and that of nontrivial ordered pairs of $L$ are denoted by $\pairs L$ and $\ntpairs L$, respectively. If $X\subseteq L$, then $\pairs X$ will stand for $X^2\cap\pairs L$.
We also need the notation $\covpairs L:=\set{\pair x y\in \pairs X: x\prec y}$ for the set of \emph{covering pairs}. 
By a \emph{quasi-colored  lattice} we mean a structure 
\begin{equation*}
\alg L=\tuple{L, \leq;\gamma;H,\nu}
\end{equation*}
where $\tuple{L;\leq}$ is a lattice, $\tuple{H;\nu}$ is a quasiordered set, $\gamma\colon \pairs L\to H$ is a surjective map called \emph{coloring}, and for all $\pair{u_1}{v_1},\pair{u_2}{v_2}\in \pairs L$, 
\begin{enumerate}[\quad \normalfont({C}1)]
\item\label{labqa} if $\bigl\langle{ \gamma(\pair{u_1}{v_1})}\,,\, {\gamma(\pair{u_2}{v_2})}\bigr\rangle\in \nu$, then $\cg{u_1}{v_1}\leq \cg{u_2}{v_2}$ and 
\item\label{labqb} if $\cg{u_1}{v_1}\leq \cg{u_2}{v_2}$, then  $\bigl\langle{ \gamma(\pair{u_1}{v_1})}\,,\, {\gamma(\pair{u_2}{v_2})}\bigr\rangle\in \nu$.
\end{enumerate}
This concept is taken from  \cite{czgprincc} or \cite{czginjlatcat}; for some earlier variants of the concept, 
see  
\init{G.\ }Gr\"atzer, \init{H.\ }Lakser, and \init{E.T.\ }Schmidt~\cite{grlaksersch}, \init{G.\ }Gr\"atzer~\cite[page 39]{grbypict}, and  \cite{czgrepres}. 
For a quasiordered set $\tuple{H,\nu}$, we let $\Theta_\nu=\nu\cap\nu^{-1}$. Then $\Theta_\nu$  is an equivalence relation, and the definition 
\begin{equation}
\pair{\blokk x{\Theta_\nu}} {\blokk y{\Theta_\nu}} \in \nu/\Theta_\nu  \overset{\textup{def}}\iff \pair x y\in \nu
\label{eqrperThnu}
\end{equation}
turns the quotient set $H/\Theta_\nu$ into an ordered set $\tuple{H;\nu}/\Theta_\nu:=\tuple{H/\Theta_\nu;\nu/\Theta_\nu}$. The importance of quasi-colored lattices is explained by the following lemma, which is a straightforward consequence of (C\ref{labqa}) and (C\ref{labqb}); see \cite[Lemma 2.1]{czgprincc} or \cite[Lemma 4.7]{czginjlatcat}.

\begin{lemma}\label{lemmaqchaszna}
If $\alg L=\tuple{L, \leq;\gamma;H,\nu}$ is a quasi-colored lattice, then $\princ L$ is isomorphic to $\tuple{H;\nu}/\Theta_\nu$.
\end{lemma}

Given a quasi-colored lattice $\alg L=\tuple{L, \leq;\gamma;H,\nu}$, a pair $\pair a b\in \covpairs L$, and a simple bounded lattice $K$, we define a new quasi-colored lattice 
\[\alg L(a,b,K)=\tuple{L(a,b,K), \leq_{a,b,K};\gamma_{a,b,K};H,\nu}
\]
as follows. To obtain $L(a,b,K)$, we insert $K$ into the prime interval $[a,b]$ such that we identify $0_K$ and $1_K$ with $a$ and $b$, respectively. This makes the meaning of $\leq_{a,b,K}$ clear. The ordered set $L(a,b,K)$ we obtain in this way is obviously a lattice. Since $\pairs L\subseteq \pairs {L(a,b,K)}$, we  can define $\gamma_{a,b,K}$ as the extension of $\gamma$ such that, for $\pair x y\in \pairs {L(a,b,K)} \setminus \pairs L$,  
\begin{equation}
\gamma_{a,b,K}(\pair x y)=
\begin{cases}
0,&\text{if }x=y,\cr
\gamma(\pair a b),&\text{if } x,y\in K\text{ and }x\neq y,\cr
\gamma(\pair a y),&\text{if } x\in K\text{ and }y\notin K,\cr
\gamma(\pair x b),&\text{if } x\notin K\text{ and }y\in K\text.
\end{cases}
\label{eqrgmmaabKdf}
\end{equation}

The straighforward \red{(but not so short)} proof of the following lemma is left to the reader
\red{(details will be given later).}

\begin{lemma}\label{lemmalabk}
 If $\alg L=\tuple{L, \leq;\gamma;H,\nu}$ is a quasi-colored lattice, then so is 
the above-defined $\alg L(a,b,K)$, \red{ provided that for all $x<a$ and $y>b$, 
$\cgi L x a = \cgi L y b=\nablaell L$.}
\end{lemma}

\begin{figure}[htc]
\centerline
{\includegraphics[scale=0.9]{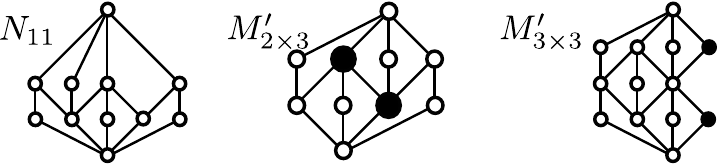}}
\caption{$M_{2\times 3}'$ and $M_{3\times 3}'$\label{figmhhvsz}}
\end{figure}

\begin{corollary}\label{corLKabsimple}
If $\pair a b\in \covpairs L$ and $K$ is a simple lattice, then $L(a,b,K)$ is also a simple lattice. 
\end{corollary}

\begin{figure}[htc]
\centerline
{\includegraphics[scale=0.9]{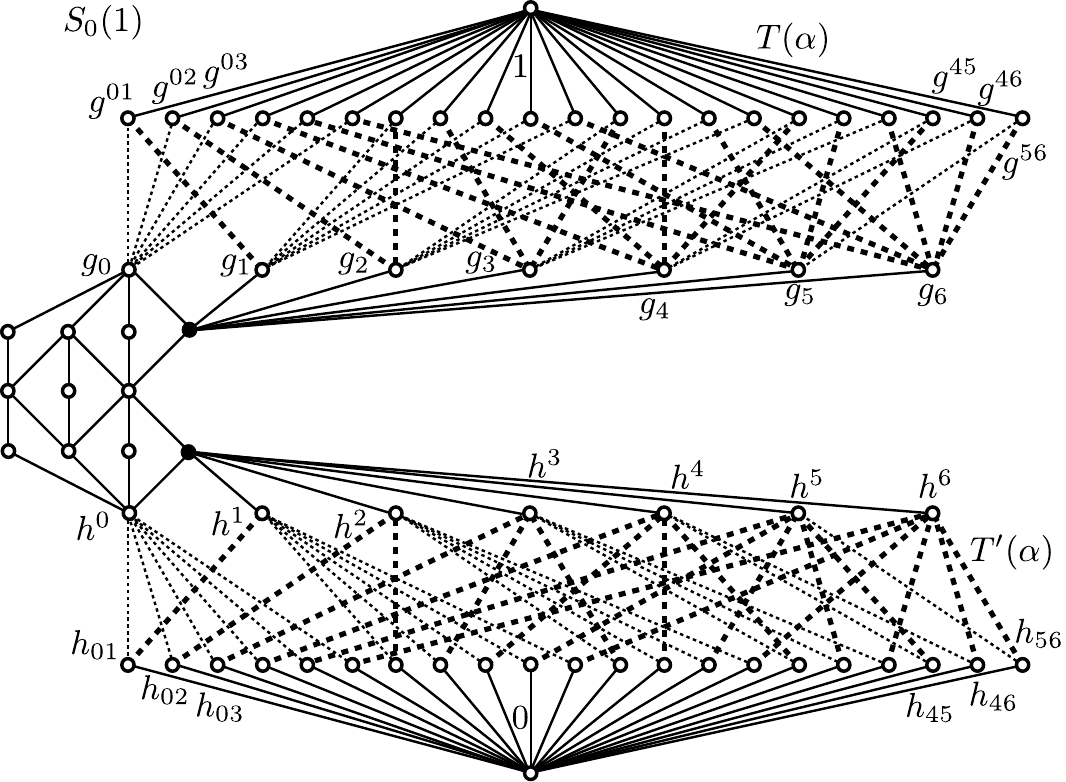}}
\caption{$S_0(1)$\label{fig-fgdZlT}}
\end{figure}

A lattice or a graph is \emph{automorphism-rigid} if its automorphism group is one-element. We are going to define a class $\{S(\alpha): \alpha$ is an ordinal number$\}$ of pairwise non-isomorphic automorphism-rigid simple lattices of length 12.
Let $\alpha$ be an ordinal number, and let 
$\atoms\alpha=\set{g_\iota: \iota<6+\alpha}$ and 
$\coatoms\alpha=\set{g^{\iota\mu}: \iota<\mu<6+\alpha}$. We agree that these two sets are disjoint from each other and from $\set{0,1}$.
On the set $T(\alpha):=\atoms\alpha\cup\coatoms\alpha\cup\set{0,1}$, we define an ordering as follows: 0 and 1 are the bottom and top elements, $\atoms\alpha$ and $\coatoms\alpha$ are the set of atoms and that of coatoms, respectively, and, for $\iota<6+\alpha$ and $\kappa<\mu<6+\alpha$, 
$g_\iota < g^{\kappa\mu}\overset{\text{def}}\iff \iota\in\set{\kappa,\mu}
$.
Similarly, let $\datoms\alpha=\set{h_{\iota\mu}: \iota<\mu<6+\alpha}$, $\dcoatoms\alpha=\set{h^\iota: \iota<6+\alpha}$, $T'(\alpha):=\datoms\alpha\cup\dcoatoms\alpha\cup\set{0,1}$ such that $\datoms\alpha$ and $\dcoatoms\alpha$ are its sets of atoms and coatoms, respectively, and 
$h_{\kappa\mu}< h^\iota  \overset{\text{def}}\iff \iota\in\set{\kappa,\mu}
$. That is, $T'(\alpha)$ is the dual of $T(\alpha)$.

Next, consider the lattice $M_{3\times 3}'$ given by Figure~\ref{figmhhvsz}. The black-filled atom and the black-filled coatom determines a principal ideal $I$ and a principal filter $F$, respectively. 
Form the Hall--Dilworth gluing of $T(\alpha)$ and $M_{3\times 3}'$  along $F$ and the principal ideal $\ideal{g_0}$. 
In the next step, form the Hall--Dilworth gluing of the lattice we have just obtained and $T'(\alpha)$ along $I$ and the principal filter $\filter{h^0}$. The lattice we obtain in this way is $S_0(\alpha)$. For $S_0(1)$, see Figure~\ref{fig-fgdZlT}.  

For $\iota<\mu<6+\alpha$, an edge of one of the forms 
$\pair{g_\iota} {g^{\iota\mu}}$, $\pair{g_\mu} {g^{\iota\mu}}$,  $\pair{h_{\iota\mu}}{h^{\iota}}$, and $\pair{h_{\iota\mu}}{h^{\mu}}$ is called an
\emph{upper left edge}, an
\emph{upper right edge}, a
\emph{lower left edge}, and a \emph{lower right edge}, respectively. (This terminology is motivated by the connection between $S_0(\alpha)$ and Frucht's graphs; see in the proof later.) The (upper and lower) left  edges are indicated by densely dotted lines in Figure~\ref{fig-fgdZlT}. The (upper and lower) right  edges are thick dotted lines, and there are also ``ordinary'' edges, the solid lines. 
We replace each upper left  edge and each lower right  edge of $S_0(\alpha)$ with a copy of the lattice $N_{11}$ from Figure~\ref{figmhhvsz}, using disjoint copies for distinct edges. Similarly, we replace each 
each lower left  edge and 
upper right  edge of $S_0(\alpha)$ with the dual $N^{(d)}_{11}$ of $N_{11}$, using disjoint copies for distinct edges again. The lattice we obtain is denoted by $S(\alpha)$.

\begin{lemma} \label{lemmaSalpha}
For every ordinal $\alpha$, $S(\alpha)$ is an
automorphism-rigid simple selfdual lattice of length $12$. Moreover, $S(\alpha)\cong S(\beta)$ iff $\alpha=\beta$.
\end{lemma}

\begin{proof} With $V:=\set{\iota: \iota<6+\alpha}$ and $E:=\set{\set{\iota,\mu}: \iota<\mu<\alpha}$, $\tuple{V;E}$ is a graph. 
Notice that $T(\alpha)$ is the Frucht graph associated with $\tuple{V;E}$; see \init{R.~}Frucht~\cite{fruchtlat} and \init{G.~}Gr\"atzer~\cite[Figure 15.1]{grbypict}. We know from \init{G.~}Gr\"atzer and \init{H.~}Lakser~\cite{grlakser} or  \init{G.~}Gr\"atzer~\cite[Page 188]{grbypict} that $T(\alpha)$ is a simple lattice. (This is why we use $6+\alpha$ rather than $\alpha$ in its definition.) Since $T'(\alpha)$, the dual of $T(\alpha)$, and $M'_{3\times 3}$ are also simple, it follows that $S_0(\alpha)$ is simple.
Finally, since $N_{11}$ and its dual are simple,  Corollary~\ref{corLKabsimple} yields that $S(\alpha)$ is a simple lattice. Since $S_0(\alpha)$ is of length 8, $S(\alpha)$ is of length 12. Also, it is a ranked lattice, that is, any two maximal chains of $S(\alpha)$ have the same number of elements. 
While the graph $\tuple{V;E}$ is encoded in $S_0(\alpha)$, the well-ordered set
$\set{\iota: \iota <6+\alpha}$ is encoded in $S(\alpha)$ as follows: the elements $h^\iota$ and $g_\iota$ can be recognized as the elements of height 4 and the elements dual height 4, respectively. Furthermore, $\iota<\mu$ 
iff the interval $[g_\iota,g_\iota\vee g_\mu]$ is isomorphic to $N_{11}$ 
iff $[g_\mu,g_\iota\vee g_\mu]\cong N^{(d)}_{11}$ 
iff $[h^\iota\wedge h^\mu,h^\iota]\cong N^{(d)}_{11}$ 
iff $[h^\iota\wedge h^\mu,h^\mu]\cong N_{11}$. Hence, if $S(\alpha)\cong S(\beta)$, then 
$\tuple{\set{\iota: \iota <6+\alpha};\leq}\cong\tuple{\set{\iota: \iota <6+\beta};\leq}$, whence $6+\alpha=6+\beta$, implying  that $\alpha=\beta$. This proves the second part of the lemma. 

Clearly, $S(\alpha)$ is a selfdual lattice. Let $f$ be an arbitrary automorphism of $S(\alpha)$.
As we have noticed above, the elements $g_\iota$ are recognized by a first-order property. Hence, 
$f(\set{g_\iota:\iota<6+\alpha}\subseteq \set{g_\iota:\iota<6+\alpha}$. In fact, we have equality here, because the same kind of inclusion holds for $f^{-1}$. However, since the well-ordering of $\set{\iota:\iota<6+\alpha}$ is encoded in the lattice, we obtain that $f$ induces an order automorphism on $\set{\iota:\iota<6+\alpha}$. It is well-known, and it follows by a straightforward transfinite induction, that $\tuple{\set{\iota:\iota<6+\alpha};<}$ is automorphism-rigid. Therefore, $f$ acts as the identity map on $\set{g_\iota:\iota<6+\alpha}$. By duality, the same holds for the set $\set{h^\iota:\iota<6+\alpha}$. Since these two sets generate $T(\alpha)$ and $T'(\alpha)$, respectively, $f$ acts identically on $T(\alpha)\cup T'(\alpha)$. In particular, the black-filled elements are fixed points of $f$, which implies that $f$ acts identically on $M_{3\times 3}'$. Consequently, so does $f$ 
on $S_0(\alpha)$. Finally, since $N_{11}$ and $N^{(d)}_{11}$ are automorphism-rigid, we obtain that $f$ 
is the identity map. Thus,  $S(\alpha)$ is automorphism-rigid.
\end{proof}

\section{A construction and completing the proof}
\label{sectconstcompl}
Besides the general case, the construction is also explained by 

\begin{example}\label{exegy} Assume that we want to represent the ordered set $P=\tuple{P;\le}$ given in Figure~\ref{figcncrTpg} and the dihedral group $G:=D_4$. First, we represent $G$ as the automorphism group of the graph $\tuple{V;E}$ given in the Figure.
\end{example}

\begin{figure}[htc]
\centerline
{\includegraphics[scale=0.9]{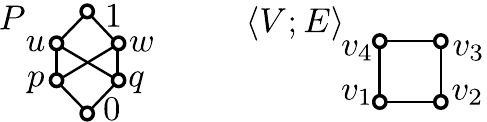}}
\caption{A small example\label{figcncrTpg}}
\end{figure}

We will frequently use the notation
$\minne P:=P\setminus\set{0,1}$.
In general, Lemma~\ref{lemmasabidussi} always allows us to take a graph $\tuple{V;E}$ whose automorphism group is isomorphic to $G$. We shall assume that $V$ is disjoint from $P$. Let $H:=P\cup V$, and consider the quasiordered set $\tuple{H;\nu}$, where  
\begin{equation*}
\nu:=\set{\pair x y\in P^2: x\leq_P y}\cup (H\times (\red{\set 1 \cup}\,\, V))\text.
\label{eqrnudef}
\end{equation*}
This means that each vertex $v\in V$ is added to $P$ as an additional largest element; $\tuple{H;\nu}$ has $1+|V|$ many largest elements. 
We let
\[
I:=J:=E\cup \set{\pair p q\in \minne P\times \minne P: p<q} \cup (\set 1\times V)\text.
\] 
Observe that $I\cup J\cup(\set 0\times H)\cup(H\times\set 1)$ generates $\nu$, that is,  \cite[(4.23)]{czginjlatcat} holds. 
Let $\aourn=\tuple{\ourn,\leq_0;\gamma_0;H,\nu}$ be the same quasi-colored lattice as $\alg L(H,I,J)$  from \cite[(4.21)]{czginjlatcat}, except that we use $M'_{2\times 3}$ rather than $M_{4\times 3}$ in its construction. This modification creates no problem, because the only reason that we used $M_{4\times 3}$ in \cite{czginjlatcat} rather than, say, $M_{2\times 3}$ was to ensure that the length of $L(H,I,J)$ is at least 5 even if $|P|=2$. As opposed to $M'_{2\times 3}$, which is automorphism-rigid, $M_{4\times 3}$ has four automorphisms; this is why the latter is not appropriate here. With $\Theta_\nu$ defined in \eqref{eqrperThnu}, $\tuple{H;\nu}/\Theta_\nu\cong P$. We know from \cite[Lemma 4.6]{czginjlatcat} that $\aourn$ is a quasi-colored lattice and it is selfdual. 
Thus,  Lemma~\ref{lemmaqchaszna} yields that $\princ{\ourn}\cong \tuple{H;\nu}/\Theta_\nu\cong P$.

\begin{figure}[htc]
\centerline
{\includegraphics[scale=0.9]{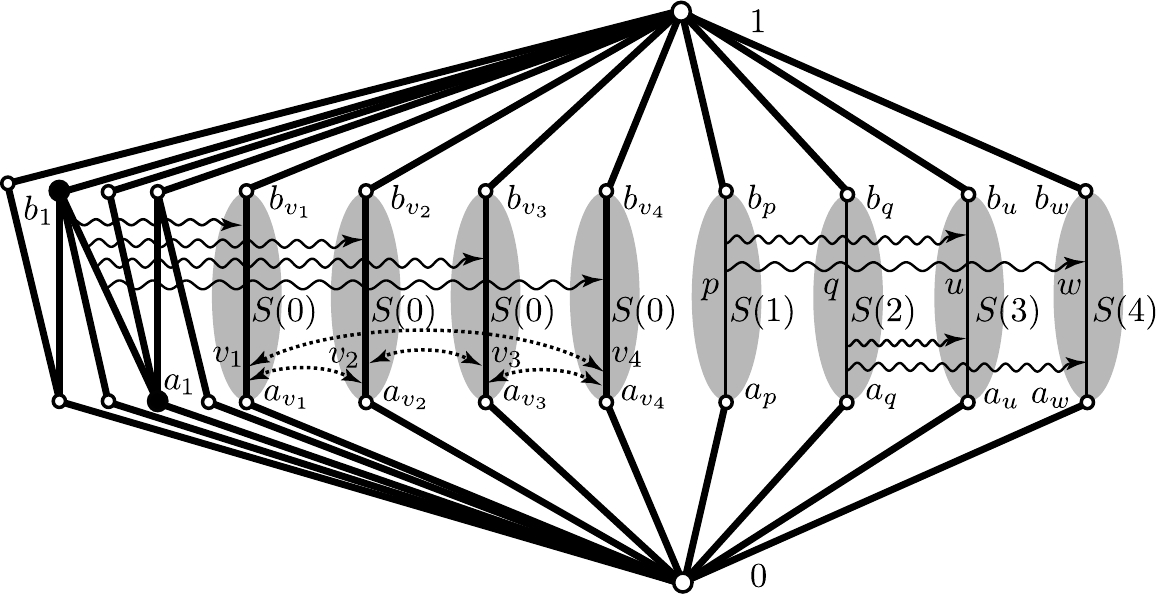}}
\caption{$\ourn$ (without the gray-filled ovals) and $\oure$\label{fig-fgUz}}
\end{figure}

\begin{example}\label{exampleKTt} For the situation described in Example~\ref{exegy} and  Figure~\ref{figcncrTpg},   we visualize $\ourn$ and  $I=J$ in Figure~\ref{fig-fgUz}. We obtain the lattice in this figure by gluing $M'_{2\times 3}$ from Figure~\ref{figmhhvsz} and the chains $\set{0\prec a_x\prec b_x\prec 1}$ for $x\in V\cup \minne P$ at their bottom and top elements.  (Disregard the gray-filled ovals $S(0),\dots, S(0), S(1),\dots S(4)$ in the figure now.) The members of $I=J$ are indicated by arrows: if $\pair x y\in I$, then there is an arrow from the prime interval $[a_x,b_x]$ to the prime interval $[a_y,b_y]$. However, we use two kinds of arrows: dotted arrows for $\pair x y\in E$ and wavy arrows otherwise.  Note that a dotted arc represents two arrows; one from left to right and another one from right to left.
As it is explained in \cite{czginjlatcat}, we obtain $\ourn$ from Figure~\ref{fig-fgUz} so that for every $\pair p q\in I=J$, we replace the corresponding arrow with the double gadget  
$\adbgad(p,q)$ given in \cite[Figure 4]{czginjlatcat}.
\end{example}

For each $p\in \minne P$, pick an ordinal number $\iota_p>0$. We assume that $\iota_p\neq \iota_q$ if $p\neq q$. Also, if $P$ is finite, then let all the $\iota_p$ be finite.
To complete the construction, we replace the prime interval $[a_p,b_p]$ with $S(\iota_p)$ for $p\in \minne P$ and we replace 
$[a_v,b_v]$ with $S(0)$ for $v\in V$. 
The $S(\iota_p)$ and all copies of the $S(0)$ are pairwise disjoint, of course. 
Using a trivial transfinite induction, it follows from Lemma~\ref{lemmalabk} that we obtain a quasi-colored lattice
$\aoure=\tuple{\oure,\leq_1; \gamma_1;H,\nu}$ in this way. 

\begin{example}\label{examplehrm} For $P$ and $G$ from Example~\ref{exegy}, Figure~\ref{fig-fgUz}
gives the lattice $\oure$ with $\princ{\oure}\cong P$ and $\aut{\oure}\cong G$ as follows. The arrows indicate gadgets, as explained in Example~\ref{exampleKTt}.
An edge $\pair x y\in\covpairs L$ (or an interval indicated by an edge) is thick iff it generates the largest congruence iff $\pair 1{\gamma(\pair x y)}\in \nu$.
 The gray-filled ovals $S(0),\dots,S(4)$ stand for the lattices defined before Lemma~\ref{lemmaSalpha}; note that $S(1)$ is derived from $S_0(1)$ given in Figure~\ref{fig-fgdZlT}.
\end{example}

Now, we are in the position to proceed as follows.

\begin{proof}[Proof of Theorem~\ref{thmmain}]
We are going to show that $\princ\oure\cong P$ and $\aut\oure\cong G$.
Since we have already seen that $\princ{\ourn}\cong P$, it follows from Lemmas~\ref{lemmaqchaszna} and \ref{lemmalabk} that 
$\princ{\oure}\cong\princ{\ourn}\cong P$. 
Hence, it suffices to deal with $\aut{\oure}$. 
We say that a subset $X$ of $\oure$ is \emph{rigid}, if the restriction of every member of $\aut{\oure}$ to $X$ is the identity map of $X$. If $f(X)\subseteq X$ for all $f\in\aut{\oure}$, then $X$ an \emph{invariant} subset. For such a subset $X$, $X=f(f^{-1}(X))\subseteq f(X)$. That is, if $X$ is an invariant subset, then $f(X)=X$ holds for all $f\in \aut{\oure}$.

Since $M'_{2\times 3}$ is automorphism-rigid and it is isomorphic to no other cover-preserving $\set{0,1}$-sublattice of $\oure$, it follows that $M'_{2\times 3}$ and, in particular, $\set{a_1,b_1}$ are rigid subsets. The elements $a_x$, $x\in V\cup \minne P$,  are characterized by the properties that $[0,a_x]$ is of length at most 2 and $a_x$ is covered by at least ${6\choose 2}=15$ elements. (This is the second reason why we used $6+\alpha$ rather than $\alpha$ in the definition of $S(\alpha)$, since an element of height 2 with nine covers need not be of the form $a_x$.) 
Therefore, taking duality also into account,
\begin{equation}
\set{a_x: x\in V\cup \minne P}\text{ and }\set{b_x: x\in V\cup \minne P}\text{ are invariant subsets.}
\label{eqrdluiTH}
\end{equation}
For distinct $p,q\in \minne P$ and $v\in V$, observe that $a_p$, $a_q$, and $a_v$ are the bottoms of $S(\iota_p)$, $S(\iota_q)$, and $S(0)$. Since $S(\iota_p)$, $S(\iota_q)$, and $S(0)$ are pairwise non-isomorphic by Lemma~\ref{lemmaSalpha},  no automorphism maps $a_p$ to $a_q$ or $a_v$. Hence, 
\begin{equation}
\set{a_p:p\in \minne P}\text{ is a rigid subset, and so is }
\set{b_p:p\in \minne P} 
\label{eqrJdGsP}
\end{equation}
by duality.
For $x\neq y\in H$, there is at most one gadget (that is, at most one arrow in Figure~\ref{fig-fgUz}) from $[a_x,b_x]$ to  $[a_y,b_y]$. If there is a gadget from  $[a_x,b_x]$ to  $[a_y,b_y]$ and  $f\in\aut\oure$, then the restriction of $f$ to $\set{a_x,b_x,a_y,b_y}$ determines its restriction to the whole the gadget. 
Since $a_x\leq b_y$ iff $x=y$, it follows that if $f(a_x)=a_y$, then $f(b_x)=b_y$. Also, $[a_x,b_x]\cong S(\iota_x)$ for $x\in \minne P$ and $[a_v,b_v]\cong S(0)$ for $v\in V$ are automorphism-rigid by Lemma~\ref{lemmaSalpha}. Putting all the above facts, including \eqref{eqrdluiTH}, and \eqref{eqrJdGsP}, together, we obtain that 
\begin{equation}
\parbox{11 cm}{$\set{a_v:v\in V}$
is an invariant subset and $f\in\aut\oure$ is 
determined by its restriction to this subset.}
\end{equation}

For distinct $x,y\in V\cup\set{1}$ , $f$ clearly preserves the property ``there is a gadget from $[a_x,b_x]$ to $[a_y,b_y]$''. But $[a_1,b_1]=\set{a_1,b_1}\subseteq M'_{2\times 3}$ is a rigid subset, so only $x,y\in V$ is interesting from this point of views. In the spirit of Figure~\ref{fig-fgdZlT}, $f$ preserves the dotted arrows, and also the absence of these arrows. Therefore, $f$ induces an automorphism of the graph $\tuple{V;E}$. Conversely, since the intervals $[a_v,b_v]$ of $\oure$, $v\in V$, are isomorphic and they are only in connection with themselves (and, in the same way, with $[a_1,b_1]$), we conclude that each automorphism of the graph induces a unique automorphism of the sublattice $\bigcup\set{[a_v,b_v]:v\in V}$ and, consequently, of $\oure$. This proves that $\aut\oure\cong{\tuple{V;E}}$, as required.
\end{proof}

\end{document}